\documentclass[10pt, final]{amsart}
\usepackage[english]{babel}
\usepackage{latexsym}
\usepackage{amssymb}
\usepackage{amscd}
\usepackage[cp850]{inputenc}
\usepackage[mathscr]{eucal}

\theoremstyle{plain}

\newtheorem{teor}{Theorem}[section]
\newtheorem{thm}{Theorem}[section]
\newtheorem{prop}{Proposition}[section]

\newtheorem{lema}{Lemma}[section]
\newtheorem{lemma}{Lemma}[section]

\theoremstyle{remark}

\newcommand{\R}{\mathbb{R}}

\newcommand{\seq}{\ensuremath{0\to Y\to X\to Z\to 0\:}}

\newcommand{\aproof}{\begin{proof}}
\newcommand{\zproof}{\end{proof}}

\newcommand\PO{{\mathrm{PO}}}

\newcommand{\FUN}{\mathsf}

\newcommand{\lop}{\curvearrowright }
\newtheorem{defn}[thm]{Definition}
\newcommand{\adef}{\begin{defn}}
\newcommand{\zdef}{\end{defn}}

\begin{document}
\title{On the Bounded Approximation Property  in Banach spaces}
\author{Jes\'us M. F. Castillo and Yolanda Moreno}
\address{Departamento de Matem\'aticas, Universidad de Extremadura, Avenida de Elvas s/n, 06011 Badajoz, Spain}
             \email{castillo@unex.es}
\address{Escuela Polit\'ecnica, Universidad de Extremadura, Avenida de la Universidad s/n, 10071 C\'aceres, Spain}
             \email{ymoreno@unex.es}
\thanks{This research has been supported in part by project MTM2010-20190-C02-01 and the program Junta
de Extremadura GR10113 IV Plan Regional I+D+i, Ayudas a Grupos de
Investigaci\'on}


\maketitle

\begin{abstract} We prove that the kernel of a quotient operator from an $\mathcal L_1$-space onto
a Banach space $X$ with the Bounded Approximation Property (BAP)
has the BAP. This completes earlier results of Lusky --case
$\ell_1$-- and Figiel, Johnson and Pe\l czy\'nski --case $X^*$
separable. Given a Banach space $X$, we show that if the kernel of
a quotient map from some $\mathcal L_1$-space onto $X$ has the BAP
then every kernel of every quotient map from any $\mathcal
L_1$-space onto $X$ has the BAP. The dual result for $\mathcal
L_\infty$-spaces also hold: if for some $\mathcal L_\infty$-space
$E$ some quotient $E/X$ has the BAP then for every $\mathcal
L_\infty$-space $E$ every quotient $E/X$ has the BAP.
\end{abstract}

\section{Preliminaries}

An exact sequence $0 \to Y \to X \to Z \to 0$ in the category of
Banach spaces and bounded linear operators is a diagram in which
the kernel of each arrow coincides with the image of the
preceding; the middle space $X$ is also called a {\it twisted sum}
of $Y$ and $Z$. By the open mapping theorem this means that $Y$ is
isomorphic to a subspace of $X$ and $Z$ is isomorphic to the
corresponding quotient. Two exact sequences $0\to Y \to X \to Z
\to 0$ and $0\to Y \to X_1 \to Z \to 0$ are said to be
\emph{equivalent} if there exists an operator $T:X\longrightarrow
X_{1}$ making commutative the diagram
$$
\begin{CD}
0 @>>>Y@>>>X@>>>Z@>>>0\\
&&\| &&@VTVV \|\\
0 @>>>Y@>>>X_{1}@>>>Z@>>>0.\\
\end{CD}$$ The classical 3-lemma (see \cite[p. 3]{castgonz}) shows that $T$ must be an isomorphism. An exact sequence is
said to split if it is equivalent to the trivial sequence $0 \to Y
\to Y \oplus Z \to Z \to 0$. There is a correspondence (see
\cite[Thm. 1.5.c, Section 1.6]{castgonz}) between exact sequences $0 \to Y \to X
\to Z \to 0$ of Banach spaces and the so-called $z$-linear maps
which are homogeneous maps $\omega: Z \lop Y$ (we use this
notation to stress the fact that these are not linear maps) with
the property that there exists some constant $C>0$ such that for
all finite sets $x_1, \dots, x_n \in Z$ one has
$\|\omega(\sum_{n=1}^N x_n) - \sum_{n=1}^N \omega(x_n)\| \leq C
\sum_{n=1}^N \|x_n\|$. The infimum of those constants $C$ is
called the $z$-linearity constant of $F$ and denoted $Z(\omega)$.

The process to obtain a $z$-linear map out from an exact sequence
$0\to Y \stackrel{i}\to X \stackrel{q}\to Z \to 0$ is the
following: get a homogeneous bounded selection $b: Z \to X$ for
the quotient map $q$, and then a linear $\ell: Z \to X$ selection
for the quotient map. Then $\omega= b - \ell$ is a $z$-linear map.
A $z$-linear map $\omega : Z \curvearrowright Y$ induces the exact
sequence of Banach spaces $0 \to Y \stackrel{j}\to Y \oplus_\omega
Z \stackrel{p}\to Z \to 0$ in which $Y \oplus_\omega X$ means the
completion of the vector space $Y \times X$ endowed with the
quasi-norm $\|(y,x) \|_\omega = \|y - \omega x\| + \|x\|$. The
$z$-linearity of $\omega$ makes this quasi-norm equivalent to a
norm (see \cite{castperdu}). The embedding is  $j(y)= (y,z)$ while
the quotient map is $p(y,z)=z$. The exact sequences
$$
\begin{CD}
0 @>>>Y@>i>>X@>q>>Z@>>>0\\
&&@| @VVTV @|\\
0 @>>>Y@>>j>Y\oplus_\omega Z@>>p>Z@>>>0
\end{CD}$$ are equivalent setting as $T:X\to  Y\oplus_\omega Z$ the
operator $T(x) = (x - \ell qx, qx)$. We will use the notation
$0\to Y \stackrel{i}\to X \stackrel{q}\to Z \to 0\equiv \omega$ to
mean that $\omega$ is a $z$-linear map associated to that exact
sequence. Two $z$-linear maps $\omega, \omega': Z\lop  Y$ are said
to be equivalent, and we write $\omega \equiv \omega'$, if the
induced exact sequences are equivalent. Two maps $\omega, \omega':
Z \lop Y$ are equivalent if and only if the difference $\omega -
\omega'$ can be written as $B +L$, where $B: Z \to Y$ is a
homogeneous bounded map and $L: Z \to
Y$ is a linear map.\\

Recall from \cite{castmoresob} the definition and basic properties
of the $z$-dual of a Banach space $X$. We define {\it the}
$z$-dual of $X$ as the Banach space $ X^z = [ Z_L(X,\R), Z(\cdot)
]$ of $z$-linear maps $\omega: X \lop \R$ such that
$\omega(e_\gamma)=0$ for a prefixed Hamel basis $(e_\gamma)$ of
$X$. The space $co_z(X)$ is defined as the closed linear span in
$(X^z)^*$ of the evaluation functionals $\delta_x: X^z \to \R$
given by $\delta_x(\omega) = \omega(x)$. Given a $z$-linear map
$\omega: X \lop Y$ and  a Hamel basis $(e_\gamma)$ we denote
$\nabla \omega$ the so-called \emph{canonical form} of $\omega$,
namely, the map

$$\nabla \omega (p) = \omega(p) - \sum \lambda_j \omega(e_j)$$where $p= \sum \lambda_j e_j$. The two basic properties of
$co_z(X)$ are displayed in the following proposition; the proof
can be found in \cite[Prop. 3.2]{castmoresob}:

\begin{prop}\label{initial} There is a $z$-linear map $\Omega_X: X \curvearrowright co_z(X)$
with the property that given a $z$-linear map $\omega: X
\curvearrowright Y$  then there exists an operator $\phi_\omega:
co_z(X) \to Y$ such that $\phi_\omega \Omega_X = \nabla \omega$
and $\|\phi_\omega\| = Z(\omega)$.
\end{prop}

It is easy to give examples of $z$-linear maps on finite
dimensional spaces with infinite-dimensional range: set $B: \R^2
\to C[0,1]$ defined by $B(e^{i\theta}) = x^\theta$, $0\leq\theta <
\pi$, and by homogeneity on the rest. It is however possible to
modify  a $z$-linear map defined on a finite-dimensional Banach
space so that the image is finite dimensional \cite[Lemma 2.3]{castmoresob}.
We need an improvement of that result:

\begin{lemma}\label{fin} Let $\Omega:X \curvearrowright Y$ be a $z$-linear
map and let $F$ be a finite dimensional subspace of $X$. Given a
finite set of points $x_1,\dots,x_n$ in the unit sphere of $F$ so
that $x_i\neq \pm x_j$ for $i\neq j$ and $\varepsilon>0$, there is
a $z$-linear map $\Omega_F: X\lop Y$ verifying:
\begin{enumerate}
\item $\Omega_F x_i = \Omega x_i$ for all $1\leq i\leq n$. \item
$\| \Omega_F - \Omega \|\leq (1+\varepsilon)Z(\Omega)$. \item
$Z(\Omega_F)\leq (3+\varepsilon)Z(\Omega)$ \item The image of $F$
by $\Omega_F$ spans a finite dimensional space.\end{enumerate}
\end{lemma}
\begin{proof} Let $\{e_\gamma\}_\gamma $ be a Hamel basis for $X$ formed by norm one vectors.
Assume that $F\subset [e_{\gamma_1}, \dots, e_{\gamma_k}]$. Fix a
finite set $A= \{x_1,\dots, x_n\}$ of elements of norm at most
$1+\varepsilon$ such that the unit ball of $F$ is contained in the
convex hull of $A$. Be sure that the set $A$ contains all
$e_{\gamma_1}, \dots, e_{\gamma_k}$ and that $x_i\neq \pm x_j$ for
$i\neq j$. Modify now $\Omega$ as follows: $\Omega_F a= \Omega a$
for all $a\in A$; if $p$ is a norm one element of $F$ and $p=\sum
\theta_ix_i$ then  $\Omega_F(\sum \theta_i x_i)= \sum \theta_i
\Omega x_i$. It is necessary to establish some order between the
convex combinations of the $x_i$ to have the value of $\Omega_F$
univocally defined and equal to $\Omega a$ at the points of $A$,
but this can be easily done in many ways. Condition (1) is
therefore fulfilled. We show (2): if $p=\sum \theta_ix_i$ is a
norm one element of $F$ then
\begin{eqnarray*} \|\Omega_F\left (\sum \theta_i x_i\right) - \Omega\left (\sum \theta_i x_i\right) \|
&=& \|\sum \theta_i \Omega x_i  - \Omega\left (\sum \theta_i x_i\right)\| \\
&\leq& Z(\Omega) \sum \theta_i \|x_i\| \\
&\leq&(1+\varepsilon) Z(\Omega).
\end{eqnarray*}
If $p$ is a point not in $F$ then $\Omega_F(p)=\Omega(p)$, and
thus (2) is proved. Observe that if $\Lambda$ is a bounded
homogeneous map then $Z(\Lambda)\leq 2 \|\Lambda\|$. This and the
previous estimate yield (3):

\begin{eqnarray*} Z(\Omega_F) = Z(\Omega_F -\Omega + \Omega)
&\leq& Z(\Omega_F -\Omega)  + Z(\Omega)\\
&\leq& 2\|\Omega_F -\Omega\| + Z(\Omega)\\
&\leq& 2(1+\varepsilon)Z(\Omega) + Z(\Omega)\\
&=&3(1+\varepsilon)Z(\Omega).
\end{eqnarray*}

Condition (4) is clear.
\end{proof}

A few facts about the connections between $z$-linear maps and the
associated exact sequences will be needed in this paper. Given an
exact sequence \seq $\equiv \omega$ and an operator $\alpha: Y\to
Y'$, there is a commutative diagram
$$
\begin{CD}
0 @>>>Y@>j>>Y\oplus_\omega Z@>p>>Z@>>>0\\
&&@V{\alpha}VV @V{\tau}VV @|\\
0 @>>>Y'@>>{j'}>Y'\oplus_{\alpha \omega} Z @>>{p'}>Z@>>>0
\end{CD}$$in which $\tau(y,z)= (\alpha y, z)$. Moreover,

\begin{lemma}\label{polemma} If one has a commutative diagram
\begin{equation}\label{po}
\begin{CD}
0 @>>>Y@>i>>X@>q>>Z@>>>0 \equiv \omega\\
&&@V{\alpha}VV @VTVV @|\\
0 @>>>Y'@>i'>>X'@>q'>>Z@>>>0\end{CD}\end{equation}
then
$$
\begin{CD}
0 @>>>Y'@>i'>>X'@>q'>>Z@>>>0 \equiv \alpha\omega\end{CD}$$
i.e., $\alpha \omega$ is a $z$-linear map associated with the
lower sequence in (\ref{po}).
\end{lemma}
\begin{proof} What one has to check is that the two sequences
$$
\begin{CD}
0 @>>>Y'@>>{j'}>Y'\oplus_{\alpha \omega} Z @>>{p'}>Z@>>>0\\
&&@| @V{\tau'}VV @|\\
0 @>>>Y'@>>i'>X' @>>{q'}>Z@>>>0
\end{CD}$$
are equivalent; and this happens via the map $\tau'(y,z) = y' +T(0,z)$.
Indeed, that the map $\tau'$ makes the diagram commutative is clear; its
continuity follows from the estimate:

\begin{eqnarray*} \|y' + T(0,z) \| &\leq& \|y' - \alpha \omega z\|
+ \|\alpha \omega z + T(0,z)\|\\
&=&\|y' - \alpha \omega z\| + \|\alpha \omega z + T(\omega z, z) + T(-\omega z,0)\|\\
&\leq &\|y' - \alpha \omega z\| + \|T(\omega z,z)\|\\
&\leq & \|y' - \alpha \omega z\| + \|(\omega z, z)\|_\omega\\
&= & \|y' - \alpha \omega z\| + \|z\|\\
&= & \|(y',z)\|_{\alpha \omega}.\end{eqnarray*}
\end{proof}

Analogously, given an exact sequence \seq $\equiv
\omega$ and an operator $\gamma: Z'\to Z$, there is a commutative
diagram
$$
\begin{CD}
0 @>>>Y@>j>>Y\oplus_\omega Z@>p>>Z@>>>0\\
&&@| @A{\tau}AA @AA{\gamma}A \\
0 @>>>Y@>>{j'}>Y\oplus_{\omega\gamma} Z' @>>{p'}>Z'@>>>0
\end{CD}$$
in which $\tau(y,z')= (y, \gamma z')$. Moreover,

\begin{lemma}\label{pblemma} If one has a
commutative diagram \begin{equation}\label{pb}\begin{CD}
0 @>>>Y@>i>>X@>q>>Z@>>>0 \equiv \omega\\
&& @| @ATAA @AA{\gamma}A\\
0 @>>>Y'@>>i'>X' @>>{q'}>Z'@>>>0
\end{CD}\end{equation}
then
$$
\begin{CD}
0 @>>>Y@>i'>>X'@>q'>>Z'@>>>0 \equiv \omega\gamma\end{CD}$$
i.e., $\omega\gamma$ is a $z$-linear map associated with the
lower sequence in (\ref{pb}).
\end{lemma}
\begin{proof} One has to check that the two sequences
$$
\begin{CD}
0 @>>>Y@>>{j'}>Y\oplus_{\omega \gamma} Z' @>>{p'}>Z'@>>>0\\
&&@| @A{\tau'}AA @|\\
0 @>>>Y@>>i'>X' @>>{q'}>Z'@>>>0
\end{CD}$$
are equivalent; and this happens via the map $\tau'(x') =(Tx'- \ell\gamma q' x',
q'x')$, where $\ell$ is a linear selection for $q$ such that for
some bounded selection $b$ for $q$ one has $\omega = b - \ell$.
Which means that $\omega \gamma$ is a $z$-linear map associated
with the lower sequence. The commutativity of the diagram is clear
and the continuity of $\tau'$ follows from the estimate:
\begin{eqnarray*} \|(Tx'- \ell\gamma q' x',
q'x') \|_{\omega \gamma} &= & \|Tx'- \ell\gamma q' x' -\omega
\gamma q'x' \| + \|q'x'\|\\
&= & \|Tx'- b \gamma q'x' \| + \|q'x'\|\\
&\leq & (\|T\|+ \|b\|\|\gamma\| +1) \|x'\|.\end{eqnarray*}
\end{proof}

As an immediate consequence we have.\\

\begin{lemma} \label{diagonales}$\;$
\begin{enumerate}
\item Given a
commutative diagram like (\ref{po}) the exact sequence
$\omega q'$ is equivalent to the exact sequence
$$\begin{CD}
0@>>> Y @>d>> Y'\oplus X @>m>> X' @>>> 0
\end{CD}$$
where $d(y)= (-\alpha y, i'y)$ and $m(y',x) = y' + Tx$.
\item Given a
commutative diagram like (\ref{pb}) the exact sequence
$i'\omega $ is equivalent to the exact sequence
$$\begin{CD}
0@>>> X' @>J>> X\oplus Z' @>Q>> Z @>>> 0
\end{CD}$$
where $J(x')= (Tx' , q'x')$ and $Q(x, z') = qx - \gamma z'$.
\end{enumerate}
\end{lemma}
\begin{proof} To prove (1) observe that there is a commutative diagram
$$
\begin{CD}
0 @>>>Y@>i>> X @>q>>Z @>>>0 \equiv \omega\\
&&@| @AA{\tau}A @AAq'A\\
0 @>>>Y@>>d>Y' \oplus X @>>{m}> X' @>>>0
\end{CD}$$where $\tau (y', x) = x$
and apply Lemma \ref{pblemma}. To prove (2) observe that there is
a commutative diagram
$$
\begin{CD}
0 @>>>Y@>i>> X @>q>>Z @>>>0 \equiv \omega\\
&&@Vi'VV @VV{\tau}V @|\\
0@>>> X' @>>J>X \oplus Z' @>>{Q}> Z @>>>0
\end{CD}$$where $\tau (x) = (x,0)$
and apply Lemma \ref{polemma}. \end{proof}

Following \cite{kaltloc}, an exact sequence $0 \to Y \to X \to Z
\to 0$ is said to locally split if its dual sequence $0 \to Z^*
\to X^* \to Y^* \to 0$ splits.

\begin{lemma} An exact sequence \seq $\equiv \omega$ (locally)
splits if and only if for every operator $\alpha:Y\to Y$ and $\gamma: Z'\to Z$  the sequence $\alpha \omega \gamma$ (locally)
splits.
\end{lemma}
\begin{proof} The sufficiency is obvious in both cases. To prove the necessity is simple for the splitting: if $\omega =B +L$ with $B$ homogeneous bounded and $L$ linear then $\alpha \omega \gamma = \alpha B\gamma + \alpha L \gamma$ with
$\alpha B \gamma$ homogeneous bounded and $\alpha L \gamma$ linear. The case of local spliting follows from this and the observation that if one has the commutative diagram

$$
\begin{CD}
0 @>>>Y@>>> X @>>>Z @>>>0 \equiv \omega\\
&&@V{\alpha}VV @VVV @|\\
0 @>>>Y'@>>> E @>>> Z @>>>0\equiv \alpha \omega\\
&&@| @AAA @AA{\gamma}A \\
0 @>>>Y'@>>> X' @>>> Z' @>>>0 \equiv \alpha \omega \gamma
\end{CD}$$
then the biduals form also a commutative diagram
$$
\begin{CD}
0 @>>>Y^{**}@>>> X^{**} @>>>Z^{**} @>>>0 \equiv \Omega\\
&&@V{\alpha^{**}}VV @VVV @|\\
0 @>>>Y'^{**}@>>> E^{**} @>>> Z^{**} @>>>0\equiv \alpha^{**} \Omega\\
&&@| @AAA @AA{\gamma^{**}}A \\
0 @>>>Y'^{**}@>>> X'^{**} @>>> Z'^{**} @>>>0 \equiv \alpha^{**} \Omega\gamma^{**}.
\end{CD}$$
Kalton shows in \cite[Thm. 3.5]{kaltloc} that an exact sequence locally splits if and only if its bidual
sequence splits. Thus, since $\Omega$ splits, so does $\alpha^{**} \Omega\gamma^{**}$.
\end{proof}

\adef A Banach space $X$ has the $\lambda$-BAP if for each finite
dimensional subspace $F\subset X$ there is a finite rank operator
$T:X\to X$ such that $\|T\|\leq \lambda$ and $T(f)=f$ for each
$f\in F$.\zdef

It is well known that in a locally splitting sequence \seq one
has: i) if $Y,Z$ have the BAP then also $X$ has the BAP
\cite{god}; ii) if $X$ has the BAP then $Y$ has the BAP \cite[Thm.
5.1]{kaltloc}.

\section{Results for $\mathcal L_1$-spaces}

We assume in what follows that $\mathcal L_1$ denotes an arbitrary
$\mathcal L_1$-space. In \cite{lusk1,lusk2} Lusky shows that when
$X$ is separable and has the BAP then the kernel of every quotient
map $\ell_1\to X$ has the BAP. Theorem 2.1 (b) of \cite{fjp}
asserts that if $q: \mathcal L_1 \to X$ is a quotient map and
$X^*$ has the BAP then both $\ker q$ and $(\ker q)^*$ have the
BAP. It is well-known that when $X^*$ has the BAP then also $X$
has the BAP, but the converse fails since there exist spaces with
basis whose dual do not have AP \cite[1.e.7(b)]{lindtzaf}.
Therefore, the missing case is to show that the kernel of an
arbitrary quotient map  $\mathcal L_1 \to X$ has the BAP when $X$
has the BAP.

\begin{lema}\label{eleuno} Let $X$ be a Banach space with the $\lambda$-BAP. Then $co_z(X)$ has the $(3\lambda + \varepsilon)$-BAP.
\end{lema}
\begin{proof} Fix a Hamel basis $(e_\gamma)_{\gamma\in \Gamma}$ for $X$ and let
$\Omega_X: X\lop co_z(X)$ the universal map appearing in
Proposition \ref{initial} verifying $\Omega_X e_\gamma=0$ for all
$\gamma\in \Gamma$. Let $\mathfrak F$ be a finite dimensional
subspace of $co_z(X)$. We can assume without loss of generality
that $\mathfrak F \subset [\Omega_X x_1, \dots, \Omega_X x_m]$.
Take $[x_1, \dots, x_m]\subset X$ and let $F$ be a finite set of
$\gamma$ for which $[x_1, \dots, x_m] \subset [e_{\gamma} : \gamma
\in F]$. Let $B_F: X\to X$ be a finite rank operator fixing
$[e_\gamma : \gamma \in F]$. Let $\Omega_{B_F(X)}$ be the version
of $\Omega_X$ verifying that the image of $B_F(X)$ is
finite-dimensional, which has moreover been done so that
$\Omega_{B_FX} x_k = \Omega_X x_k$ and $\Omega_{B_FX} e_\gamma =
\Omega_X e_\gamma$ for all $1\leq k\leq m$ and all $\gamma \in F$.
By the properties  of $co_z(X)$ there is an operator $\phi_F:
co_z(X)\to co_z(X)$ such that

$$\phi_F\Omega_X = \nabla (\Omega_{B_F(X)}B_F).$$

Given any $p\in X$, if $p=\sum_j \lambda_j e_j$ then
$$
\nabla (\Omega_{B_F(X)}B_F)(p) = \Omega_{B_F(X)}B_Fp - \sum_j
\lambda_j \Omega_{B_F(X)}B_F e_j$$ and therefore the image of
$\nabla (\Omega_{B_F(X)}B_F)$ spans a finite dimensional space.
This means that also the range of $\phi_F$ is finite dimensional.
Moreover, $\phi_F$ fixes $\mathfrak F$ since for $\Omega_X x_k,
1\leq k\leq m$ one has that if $x_k = \sum \lambda_i e_{\gamma_i}$
with $\gamma_i \in F$ one has

\begin{eqnarray*} \phi_F(\Omega_X x_k) &=& \nabla (\Omega_{B_F(X)} B_F) (x_k)\\
&=& \Omega_{B_F(X)} B_F(x_k) - \sum \lambda_i \Omega_{B_F(X)} B_F
e_{\gamma_i}\\ &=& \Omega_{B_F(X)} x_k - \sum \lambda_i
\Omega_{B_F(X)} e_{\gamma_i}\\
&=&\Omega_X x_k
\end{eqnarray*}

Finally, $$\|\phi_F\| = Z(\nabla (\Omega_{B_F(X)} B_F) =
Z(\Omega_{B_F(X)} B_F) \leq Z(\Omega_{B_F(X)})\| B_F\| \leq
(3+\varepsilon)Z(\Omega_X)\lambda.$$
\end{proof}

\begin{teor}\label{alluno} Let $0 \to Y \to \mathcal L_1\to X\to 0$ be an exact
sequence in which $X$ has the BAP. Then $Y$ has the BAP.
\end{teor}
\begin{proof} The universal property of the $co_z(\cdot)$  construction mentioned in
Proposition \ref{initial} yields a commutative diagram
$$\begin{CD}
0@>>> co_z(X) @>>>  \Sigma @>>> X @>>> 0 \equiv \Omega_X\\
 &&@V{\phi}VV @VVV @| \\
 0@>>> Y @>>>  \mathcal L_1 @>>p> X @>>> 0.
\end{CD}
$$Therefore, by virtue of Lemma \ref{diagonales} (1) there is an exact sequence
$$\begin{CD}
0@>>> co_z(X) @>>>  \Sigma \oplus Y @>>> \mathcal L_1
@>>> 0 \equiv \Omega_X p.
\end{CD}
$$This sequence locally splits: the dual sequence splits since the dual of an $\mathcal L_1$-space is injective. Thus, $\Sigma \oplus Y$ has the BAP, as well
as both $\Sigma$ and $Y$.
\end{proof}

\noindent \textbf{Question.} Does a similar result hold for other
well-known variations of the BAP such as the \emph{commuting
bounded approximation property} (CBAP), the \emph{uniform
approximation property} (UAP) or the existence of \emph{finite
dimensional decomposition} (FDD)? The previous proof cannot be
translated to cover the case of the CBAP or FDD since these
properties do not pass to complemented subspaces \cite{casa};
also, the result cannot be translated to the case of the UAP since
Lemma \ref{fin} spoils the estimate on the dimension of the
required finite dimensional operator.\\

The phenomenon described in the proposition --all the kernels of
all the quotient maps from an $\mathcal L_1$-space onto a space with the
BAP  have the BAP-- is part of a general stability result

\begin{prop}\label{eleunouno}  Given two exact sequences\begin{equation}\label{uno}\begin{CD}
0@>>> Y @>>>  \mathcal L_1 @>>> Z @>>> 0 \\
 &&&&&&\Vert \\
 0@>>> Y' @>>>  \mathcal L_1'   @>>> Z @>>> 0
 \end{CD}\end{equation}
Then $Y$ has the BAP if and only if $Y'$ has the BAP.
\end{prop}
\begin{proof} Observe that if one has a commutative diagram
formed by short exact sequences
 $$\begin{CD}
&& &&\omega'q&& \omega'\\
 && &&|\|&&|\|\\
  && &&0&&0 \\
  && &&@AAA @AAA \\
0@>>> Y @>>>  X @>q>> Z @>>> 0 \equiv \omega\\
 &&\Vert &&@AAA @AA{q'}A \\
 0@>>> Y @>>>  E @>>> X' @>>> 0 \equiv \omega q'\\
 &&&& @AAA @AAA \\
 &&&& Y'&=& Y' \\
   && &&@AAA @AAA \\
     && &&0&&0
\end{CD}
$$in which we assume that: \begin{enumerate} \item $\omega q'$ and $\omega'q$ locally
split. \item $X$ has the BAP \item $Y'$ has the
BAP.\end{enumerate} then also $Y$ has the BAP: Since $\omega'q$
locally split and both $X,Y'$ have the BAP, then $E$ has the BAP
\cite{god}, see also \cite[Thm. 7.3.e]{castgonz}. Since $\omega
q'$ locally splits, then $Y$ has the BAP.  Apply now this schema
to diagram (\ref{uno}) depicted in the form

$$\begin{CD}
&& &&\omega'q&& \omega'\\
 && &&|\|&&|\|\\
  && &&0&&0 \\
  && &&@AAA @AAA \\
0@>>> Y @>>>  \mathcal L_1 @>q>> Z @>>> 0 \equiv \omega\\
 &&\Vert &&@AAA @AA{q'}A \\
 0@>>> Y @>>>  E @>>> \mathcal L_1' @>>> 0 \equiv \omega q'\\
 &&&& @AAA @AAA \\
 &&&& Y'&=& Y' \\
   && &&@AAA @AAA \\
     && &&0&&0
\end{CD}
$$and recall that both $\omega'q$ and $\omega q'$ locally split since the quotient is an $\mathcal L_1$-space.\end{proof}

The following lemma, rather its dual version, will be required to
work with $\mathcal L_\infty$-spaces; we include its proof for the sake of
completeness.

\begin{lemma}\label{ne}  Given an exact sequence $0 \to Y \to X\to \mathcal L_1
\to 0$, in which $X$ has the BAP then the space $Y$ has the BAP.
\end{lemma}
\begin{proof} General properties of $\ell_1(\Gamma)$ spaces yield a commutative diagram
$$\begin{CD}
0 @>>> K  @>>> \ell_1(\Gamma) @>>> \mathcal L_1 @>>>0 \equiv \omega\\
&&@VVV @VVV @|\\
0 @>>> Y @>>> X @>>q> \mathcal L_1  @>>>0.
\end{CD}
$$By Lemma \ref{diagonales} (1) the exact sequence associated to
$\omega q$ has the form $$\begin{CD} 0 @>>> K  @>>> Y \oplus \ell_1(\Gamma) @>>> X
@>>>0.
\end{CD}
$$
This sequence locally splits since $\omega$
locally splits. The space $K$ must have the BAP and, thus, if $X$ has the BAP then $Y
\oplus \ell_1(\Gamma)$ has the BAP, which implies that also $Y$ has
the BAP.
\end{proof}

\section{Dual results for $\mathcal L_\infty$-spaces and $z$-duals}

We assume in what follows that $\mathcal L_\infty$ denotes an
arbitrary $\mathcal L_\infty$-space. For the case of $\mathcal
L_\infty$-spaces the result (and proof) of Figiel, Johnson and Pe\l
czy\'nski \cite[Thm. 2.1.(a)] {fjp} is already optimal:

\begin{prop}\label{fjp} Let  $0\to Y \to \mathcal L_\infty \to Z \to 0$ be an exact sequence in which $Y$ has the BAP. Then $Z$ has the
BAP.\end{prop}

Let us show that the dual result of Proposition \ref{alluno} for
$\mathcal L_\infty$-spaces also holds. We need a lemma, dual of Lemma \ref{ne}.

\begin{lemma}\label{new} Given an exact sequence $0 \to \mathcal L_\infty\to X\to
Z\to 0$ in which $X$ has the BAP then the space $Z$ has the BAP.
\end{lemma}
\begin{proof}
By the injectivity properties of $\ell_\infty(\Gamma)$ spaces,
there is a commutative diagram
$$\begin{CD}
0 @>>> \mathcal L_\infty  @>>> \ell_\infty(\Gamma) @>>>
\ell_\infty(\Gamma)/\mathcal L_\infty
@>>>0 \equiv \omega \\
&&@| @AAA @AAA\\
0 @>>> \mathcal L_\infty @>>\jmath> X @>>> Z @>>>0.
\end{CD}
$$By Lemma \ref{diagonales} (2)  the exact sequence associated to $\jmath \omega$ has the form
$$\begin{CD}
0 @>>> X @>>> \ell_\infty(\Gamma)\oplus Z @>>>
\ell_\infty(\Gamma)/\mathcal L_\infty @>>>0.
\end{CD}
$$This sequence locally splits since $\omega$
locally splits. The space $\ell_\infty(\Gamma)/\mathcal L_\infty
$, as a quotient of two $\mathcal L_\infty$-spaces, is an
$\mathcal L_\infty$-space, hence it has the BAP. Thus, if $X$ has
the BAP then $\ell_\infty(\Gamma)\oplus Z$ has
the BAP, which implies that also $Z$ has the BAP.
\end{proof}

Recall that the fact that both $Y,X$ have the BAP in a locally
splitting sequence $\seq$ does not imply that $Z$ has the BAP:
indeed, every separable Banach space $Z$ is a quotient $X^{**}/X$
in which both $X, X^{**}$ have a basis \cite{lindjames}. Taking as
$Z$ a space without BAP provides the example since every sequence
$0 \to X\to X^{**}\to Z\to
0$ locally splits.\\

We are ready to show the dual of Proposition \ref{eleunouno}: if
for some $\mathcal L_\infty$-space the quotient $\mathcal
L_\infty/X$ has the BAP the same happens for all $\mathcal
L_\infty$-spaces.

\begin{prop}\label{eleinfinito}  Given two exact sequences\begin{equation}\label{infty}\begin{CD}
0@>>> Y @>>>  \mathcal L_\infty @>>> Z @>>> 0 \\
 &&\Vert \\
 0@>>> Y @>>>  \mathcal L_\infty'   @>>> Z' @>>> 0 \\
 \end{CD}
\end{equation}
Then $Z$ has the BAP if and only if $Z'$ has the BAP.
\end{prop}
\begin{proof} Draw the sequences forming a commutative diagram
$$\begin{CD}
 &&\omega'&& i\omega'\\
 &&|\|&&|\|\\
 &&0&&0 \\
 &&@VVV @VVV \\
0@>>> Y @>i>>  \mathcal L_\infty @>>> Z @>>> 0 \equiv \omega\\
 &&@Vi'VV @VVV @| \\
 0@>>> \mathcal L_\infty' @>>>  \PO @>>> Z @>>> 0 \equiv i'\omega\\
&& @VVV @VVV \\
&& Z'&=& Z' \\
 &&@VVV @VVV \\
 &&0&&0;
\end{CD}
$$and assume that $Z'$ has the BAP. Since $i\omega'$ locally splits, $\PO$ must
have the BAP and thus Lemma \ref{new} applies to conclude that
also $Z$ has the BAP. \end{proof}

The following result can be considered the nonlinear  version  of
Proposition \ref{fjp}, and the dual of ``$X$ has the BAP implies
$co_z(X)$ has the BAP". It is somehow surprising since $X^*$
apparently disappears in the construction of $X^z$.

\begin{prop}\label{zdual} If $X^*$ has the BAP then also $X^z$  has
the BAP.
\end{prop}
\begin{proof} Let $\FUN Z(X,\R)$ denote the space of $z$-linear maps, $\FUN B(X,\R)$ the space of bounded homogeneous maps,
$\FUN L(X,\R)$ the space of linear maps and, as usual, $\mathfrak
L(X,\R) =X^*$ is the space of linear continuous maps on $X$. Since
$\FUN Z(X,\R) = {\FUN B}(X,\R) + {\FUN L}(X,\R)$ and $\FUN
B(X,\R)\cap \FUN L (X,\R) = \mathfrak L(X,\R)$ the diamond lemma
applied to
$$\begin{CD} &\FUN Z(X,\R)\\ \nearrow&& \nwarrow\\ \FUN B(X,\R)&&&&
\FUN L(X,\R)\\\nwarrow& &\nearrow\\ &\mathfrak L(X,\R)\end {CD}$$
yields $\FUN Z(X,\R)/\FUN L(X,\R) = \FUN B(X,\R) /\mathfrak
L(X,\R)$. But observe that, algebraically speaking, $\FUN
Z(X,\R)/\FUN L(X,\R)$ is the space $Z_L(X,\R)$; moreover, the
uniform boundedness principle for exact sequences
\cite{cabecastuni} yields that they are also isomorphic as Banach
spaces. Hence
$$X^z = Z_L(X,\R)\simeq \FUN B(X,\R) / X^*.$$

In the same form as the space of bounded functions on $X$ is
$\ell_\infty(X)$, the space $\FUN B(X,\R)$ is $\ell_\infty(S^+)$,
where $S^+$ is a half of the unit sphere (i.e., a subset of the
unit sphere with the property that for every norm one $x$,either
$x$ or $-x$ is in $S^+$.) Therefore, Proposition \ref{fjp} and the
isomorphism
$$X^z \simeq \ell_\infty(S^+) / X^*$$
complete the proof.
\end{proof}

The referee has pointed out the question of whether ``$co_z(X)$
has the BAP implies $X$ has the BAP", as it occurs in the case of the
Lipschitz-free space of Godefroy and Kalton \cite{godkalt}. This
is a tough question: By the results in this paper, if $co_z(X)$
has the BAP then the kernel of every quotient map $q: \mathcal L_1 \to
X$ has the BAP. But there are Banach spaces $Z$ with the BAP which are
subspaces of some $\mathcal L_1$-space so that $\mathcal L_1/Z$
does not have the BAP. Indeed, Prof. Szankowski has kindly informed us
the classical Enflo-Davie example provides a subspace $S$ of
$c_0$ without AP for which he proved that $c_0/S$ has the BAP (and
then, all quotients $\mathcal L_\infty/S$ have the BAP). Actually
$c_0/S$ is isomorphic to $c_0( {\ell_2}^{2^k})$ and thus its dual
$\ell_1( {\ell_2}^{2^k})$ also has the  BAP. Thus, there is an exact
sequence
$$\begin{CD}
0@>>> \ell_1( {\ell_2}^{2^k}) @>>> \ell_1 @>>> S^*@>>> 0
\end{CD}$$
in which $S^*$ fails to have the BAP. It is likely that
$co_z(S^*)$ has the BAP while $S^*$ does not, but we could not
prove it. Since every infinite dimensional $\mathcal L_1$-space
contains a complemented copy of $\ell_1$ one has

\begin{lema} Every infinite dimensional $\mathcal L_1$-space contains a
subspace $R_1$ with the BAP so that $\mathcal L_1/R_1$ does not have the BAP.\end{lema}

On the other hand, the space $\ell_p$ admits for $1<p<2$ a
subspace without the BAP \cite{szanksub}, so $\ell_p$ admits a
quotient without BAP for $2<p<+\infty$ $\ell_p$. Since $\ell_p$ is
a quotient of $C[0,1]$ for $2<p<+\infty$, the existence of
subspaces $J$ of $C[0,1]$ such that $C[0,1]/J$ does not have the
BAP is clear (this short line was mentioned to us by Bill
Johnson). In particular, $J$ does not have the BAP and no quotient
$\mathcal L_\infty/J$ can have the BAP. The existence of $J$ means
that the claim \cite[Thm. 7.1]{castmoreproc} \emph{``Every separable
Banach space $X$ can be embedded into some $\mathcal L_\infty$
space in such a way that [they verify some additional properties
and] $\mathcal L_\infty/X$ has the BAP"} is wrong. Thus, Zippin's
result --\emph{every separable Banach space $X$ can be embedded
into some $Z(X)$ in such a way that [they verify the same
additional properties as before and] $Z(X)/X$ has FDD }
\cite{zippill} cannot be (easily) improved.

\end{document}